\documentclass[leqno]{aadmbook}
\usepackage{amsthm}
\usepackage{amsfonts}
\usepackage{amsmath}
\usepackage{latexsym,amsfonts,mathptm}
\usepackage{epsfig,float,graphicx}
\restylefloat{figure} 
\def\dj{d\kern-0.4em\char"16\kern-0.1em}
\def\Dj{\mbox{\raise 0.3ex\hbox{-}\kern-0.38em D}}

\newtheorem{thm}{Theorem}
\newtheorem{lem}{Lemma}
\newtheorem{prp}{Proposition}

\newtheorem{df}{Definition}

\newtheorem{exm}{Example}

\def\ds{\displaystyle}
\def\dzn{,\kern-0.1em,}

\input cyracc.def
 
\def\be{\begin{equation} }
\def\ee{\end{equation} }
\def\bfl{\begin{flushleft} }
\def\efl{\end{flushleft} }
\def\bfr{\begin{flushright} }
\def\efr{\end{flushright} }
\def\bc{\begin{center}}
\def\vs*{\vspace*}
\def\hs*{\hspace*}
\def\ec{\end{center}}
\def\beq{\begin{eqnarray}}
\def\eeq{\end{eqnarray}}

\def\ben{\begin{enumerate}}
\def\een{\end{enumerate}}
\def\bit{\begin{itemize}}
\def\eit{\end{itemize}}

\begin{document}


\oddsidemargin 16.5mm
\evensidemargin 16.5mm

\thispagestyle{plain}



\vspace{5cc}
\begin{center}
{\large\bf  THE  STRUCTURE OF  THE  2-FACTOR TRANSFER DIGRAPH COMMON FOR
 RECTANGULAR, THICK CYLINDER AND MOEBIUS STRIP  GRID GRAPHS
\rule{0mm}{6mm}\renewcommand{\thefootnote}{}
\footnotetext{\scriptsize 2010 Mathematics Subject Classification.
05C38, 05C50, 05A15, 05C30, 05C85.

\rule{2.4mm}{0mm}Keywords and Phrases:  2-factor, Transfer matrix, Thick grid cylinder,  Moebius strip}}

\vspace{1cc} {\large\it   Jelena \Dj oki\' c, Ksenija Doroslova\v
cki\footnote{corresponding author},  Olga Bodro\v{z}a-Panti\'{c}}

 \begin{center}{\it Dedicated to the memory of a distinguished scientist and colleague, \\ Ratko To\v si\' c (1942 - 2022)} \end{center}

\vspace{1cc}

\parbox{24cc}{{\small
In this paper, we prove that all but one of the components   of   the  transfer digraph ${\cal D}^*_m$
needed for the  enumeration of 2-factors  in  the rectangular, thick cylinder and  Moebius strip grid graphs
of the fixed width $m$ ($m \in N$) are bipartite digraphs
and that their  orders  could be expressed in term of  binomial coefficients.
In addition, we prove that  the set of vertices of each component consists of all the  binary  $m$-words for which
the difference of numbers of zeros in odd and even positions is constant.
}}
\end{center}

\vspace{1cc}

\vspace{1.5cc}
\begin{center}
{\bf  1.  INTRODUCTION}
\end{center}
\label{sec:intro}

\vspace*{4mm}

Robotic and biochips technology trends  actualize the problem of  generating and enumeration of  Hamiltonian paths in grid graphs \cite{LCK,NW}.
 The counting of Hamiltonian cycles  on   specific  grid graphs was the subject of interest in
\cite{BKDP1}-\cite{BT94}, \cite{Kar}, \cite{KaP}, \cite{P} and \cite{VZB}.
 The transfer matrix approach  has been  proven to be the  most suitable  for this
  and similar problems \cite{EJ,J1,KJ1}. Namely, the specificity  of the considered graphs
  is reflected in possibility of grouping their vertices in columns that are  suitable  for coding
as words of fixed length over some alphabet.
Whether  the subgraph  induced by the vertices from the  same column is the path $P_m$ or cycle $C_m$  ($m$ is the number of vertices in columns) we refer to these grid graphs as {\em linear} (the rectangular, thick cylinder and  Moebius strip grid graphs)
 or  {\em circular} ones (thin cylinder, torus  and  Klein bottle grid graphs).
 In this paper we deal with the former case. In the  latter case the coding words are circular and
 the  reader interested in that topic is referred to \cite{DjBD3}.

\begin{df} \label{def:grafovi}
{\bf {\em The Rectangular (grid) graph}} $RG_m(n)$, {\bf {\em thin (grid) cylinder}}  $TnC_m(n)$ and  {\bf {\em  thick (grid) cylinder}} $TkC_m(n)$ ($m,n \in N$) are
 $P_m \times P_{n}$, $ C_m \times P_n$ and  $P_m \times C_n$, respectively. \\
 {\bf {\em The  Moebius strip}} $MS_m(n)$  is obtained from $RG_m(n+1) = P_m \times P_{n+1}$
  by identification  of  corresponding vertices from the first and  last  column in the opposite direction and without duplicating edges.
  The value $m$ is called the {\bf {\em  width}} of the grid graph.
  \end{df}
The thick grid cylinder $TkC_m(n)$  can be also obtained from  $RG_m(n+1) = P_m \times P_{n+1}$
  by identification  of  corresponding vertices from the first and  last  column  (in the  same direction) and without  duplicating edges (see  Figure~\ref{SCiMS}).

\begin{figure}[htb]
\begin{center}
\includegraphics[width=5in]{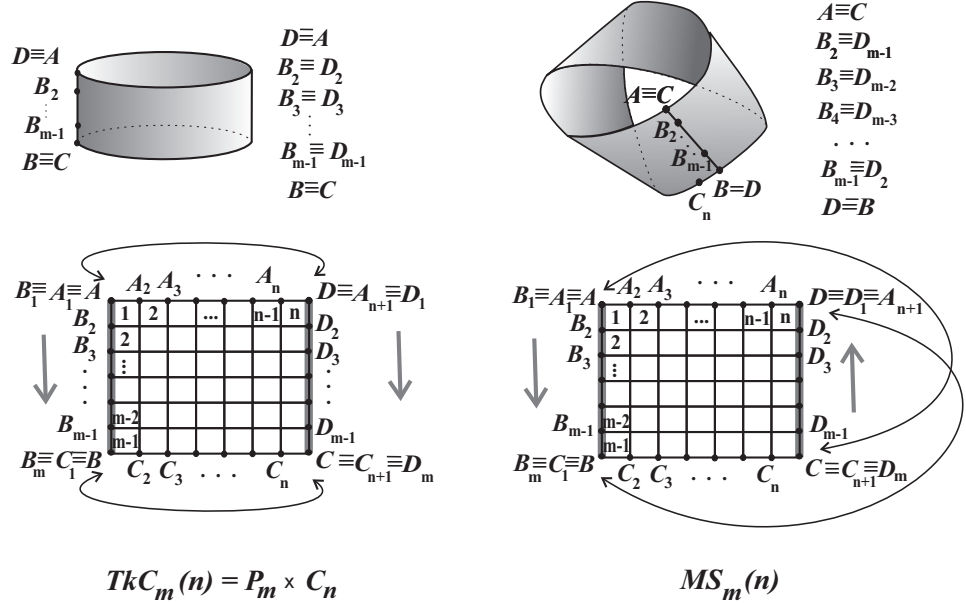}
\\ \ \vspace*{-18pt}
\end{center}
\caption{ The thick grid cylinder $TkC_m(n)$ and the  Moebius strip $MS_m(n)$}
\label{SCiMS}
\end{figure}

 In  recent papers \cite{BKDP1}-\cite{BKP} dealing with Hamiltonian cycles in the rectangular grid graph, thin
 and   thick  cylinder and  their triangular variants
 some open   questions have occurred.
 For example, the numbers of the contractible and non-contractible Hamiltonian cycles  for thin cylinder graph are  asymptotically equal (when $n \rightarrow \infty$) \cite{BKP} and the same is valid  for its triangular variant \cite{BKDP2}.
For the thick grid cylinder the  contractible Hamiltonian cycles are more numerous than the  non-contractible ones iff  $m$ is even \cite{BKDjDP}.
The coefficient for dominant eigenvalue  for non-contractible HCs is equal to $1$ (computational data for $m \leq 10$) \cite{BKDP1}.
Also, positive dominant characteristic root  for contractible HCs in a thick grid cylinder is equal to the same one associated with rectangular grid graph (computational data for $m \leq 10$) \cite{BKDP1,BPPB}.

Motivated to approach  more closely  to the  answers on  these questions we started  with  the  investigation of   2-factors  in these graphs as a generalization of  Hamiltonian cycles \cite{DjBD}. Additionally, we expanded our research to the new class of grid  graphs -  Moebius strips  $MS_m(n)$. We   wondered if the same or similar properties   related to HCs   would  remain valid for 2-factors or not, and wanted to see if  some conclusions for 2-factors could help in  proving the mentioned conjectures  for HCs. For example, the property of the coefficient for dominant eigenvalue  for non-contractible HCs appeared also
in case of 2-factors for both $TkG_m(n)$ and $MS_m(n)$.

A spanning 2-regular subgraph of a graph is  called a  \emph{ $2$-factor}. Obviously, it is  a union of disjoint cycles.
In  Figure~\ref{Adinkras},   the boundary of the (gray) figure, known as {\em Adinkra}, consists of $7$ cycles and represents a 2-factor of the rectangular grid graph $RG_{14}(14) = P_{14} \times P_{14}$. (In theoretical physics, adinkras  are  geometric objects that encode mathematical relationships between supersymmetric particles. The name ``adinkra'' is linked to  West African symbols that represent wise sayings \cite{G}.)

\begin{figure}[htb]
\begin{center}
\includegraphics[width=1.5in]{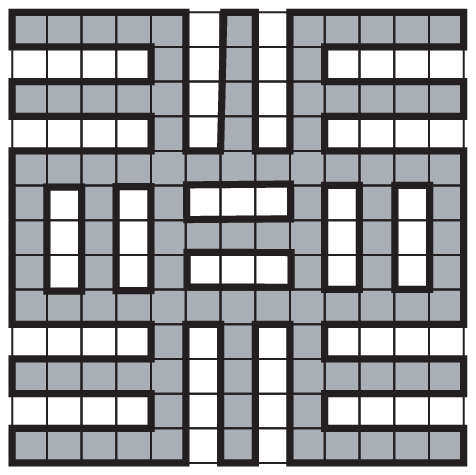}
\\ \ \vspace*{-18pt}
\end{center}
\caption{ Adinkra ``nea onnim no sua a, ohu'' (``he who does not know can become knowledgeable through learning'')}
\label{Adinkras}
\end{figure}

For a given 2-factor of a grid graph $G$ there exist six possible
situations in any vertex shown in  Figure~\ref{CvorniKod}. Namely, for any vertex $v \in V(G)$ exactly two edges of the considered 2-factor (bold lines) are incident to $v$.
The letters  attached to these situations are called {\em code letters}

\begin{figure}[htb]
\begin{center}
\includegraphics[width=3in]{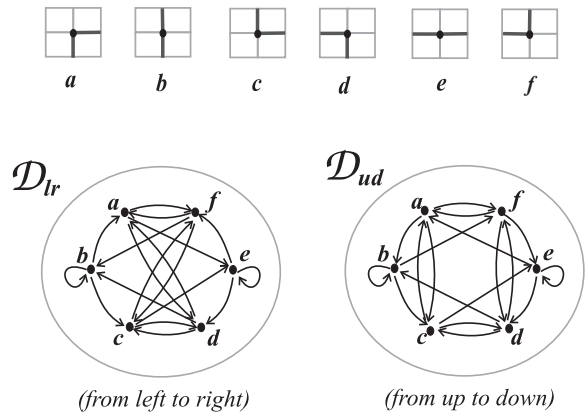}
\\ \ \vspace*{-18pt}
\end{center}
\caption{ Six possible
situations  in any vertex for a given 2-factor (above); the  digraphs ${\cal D}_{ud}$ and  ${\cal D}_{lr}$ (below)}
\label{CvorniKod}
\end{figure}

\begin{df} \label{def:2-factor}
For a given 2-factor of a linear grid graph $G$  of width $m$  and with $m \cdot n$ vertices ($m,n \in N$),
the {\bf {\em code matrix}} $ \ds \left[ \alpha_{i,j}\right]_{m \times n}$  is
a matrix of order $m \times n$ with entries from $\{ a,b, c, d, e,f \}$ where
$\alpha_{i,j}$ is the code letter for the $i$-th vertex in $j$-th
column of $G$.
\end{df}

The possibility that two code letters appear as neighbors in the same column  (row) of the code matrix is shown in the
 auxiliary digraph  ${\cal D}_{ud}$ (${\cal D}_{lr}$) in   Figure~\ref{CvorniKod}.
The code matrix  $ \ds \left[ \alpha_{i,j}\right]_{m \times n}$ for a given 2-factor of a linear grid graph $G$
 has the following properties \cite{DjBD}:

\begin{enumerate}
\item \textbf{Column conditions:} \ For every fixed \ $j$  \ ($1 \leq j \leq n$),

\begin{enumerate}
 \item \ the ordered pairs \ $ (\alpha_{i,j}, \alpha_{i+1,j})$, \ where \ $1 \leq i
\leq m-1$, \ must be arcs in the digraph \ ${\cal D}_{ud}$.

\item  $ \alpha_{1,j} \in \{ a, d, e \}$ \ and
\ $\alpha_{m,j} \in \{ c, e, f \}$.
             \end{enumerate}

\item \textbf{Adjacency of column condition:} \  For every  fixed  $j$, where  \ $1 \leq j \leq n-1 $,
 the ordered pairs \ $ (\alpha_{i,j}, \alpha_{i,j+1})$, \ where \ $1
\leq i \leq m$, \ must be arcs in the digraph \ ${\cal D}_{lr}$.

\item \textbf{First and Last Column conditions:}
\begin{enumerate}
 \item
If $G= RG_{m}(n)$, then
the alpha-word of the first  column consists of the letters from the
set \ $\{ a, b, c \}$
 and of  the last column of the letters
  from the set \ $\{  b, d, f \}$.

\item
If $G= TkC_{m}(n)$,  then
 the ordered pairs \ $ (\alpha_{i,n}, \alpha_{i,1})$, \ where \ $1
\leq i \leq m$, \ must be arcs in the digraph \ ${\cal D}_{lr}$.

\item
If $G= MS_{m}(n)$, then
 the ordered pairs \ $ (\overline{\alpha}_{i,n}, \alpha_{m -i+1,1})$, $1 \leq i \leq m$,  must be arcs in the digraph \ ${\cal D}_{lr}$,  \\
where
$\overline{a} = c, \overline{b} = b, \overline{c} = a$,
 $ \overline{d} = f, \overline{e} = e$ and  $\overline{f} = d$ (the adequate label obtained
 by applying reflection  symmetry with the horizontal  axis as its line of symmetry).
 \end{enumerate}
\end{enumerate}

This enables that the counting of such code matrices (in fact 2-factors) is  reduced to  the counting of directed walks
in an auxiliary  digraph  $ {\cal D}_{m}
\stackrel{\rm def}{=} $  $(V({\cal D}_{m}), E({\cal D}_{m}))$,  common  for all linear  graphs.
The set of its  vertices  \ $ V({\cal D}_{m}) $  \ consists of all possible words $\alpha_{1,j}\alpha_{2,j} \ldots \alpha_{m,j}$ over  alphabet $ \{ a,b,c,d,e,f \}$ (called  {\em alpha-words}) which fulfill  {\em  Column conditions}. An arc $(v,u) \in E({\cal D}_{m})$
 joins  $ v = \alpha_{1,j}\alpha_{2,j} \ldots \alpha_{m,j}$ to  \ $ u= \alpha_{1,j+1}\alpha_{2,j+1} \ldots \alpha_{m,j+1}$, i.e.  $ v \rightarrow u $
   iff the {\em Adjacency of column condition} is satisfied for the ordered pair $(v,u)$
(i.e. the vertex $v$  can be the previous column for the vertex $u$ in  the  code  matrix  $ \left[ \alpha_{i,j}\right]_{m \times n}$ for a 2-factor of $G$).

\begin{figure}[htb]
\begin{center}
\includegraphics[width=3in]{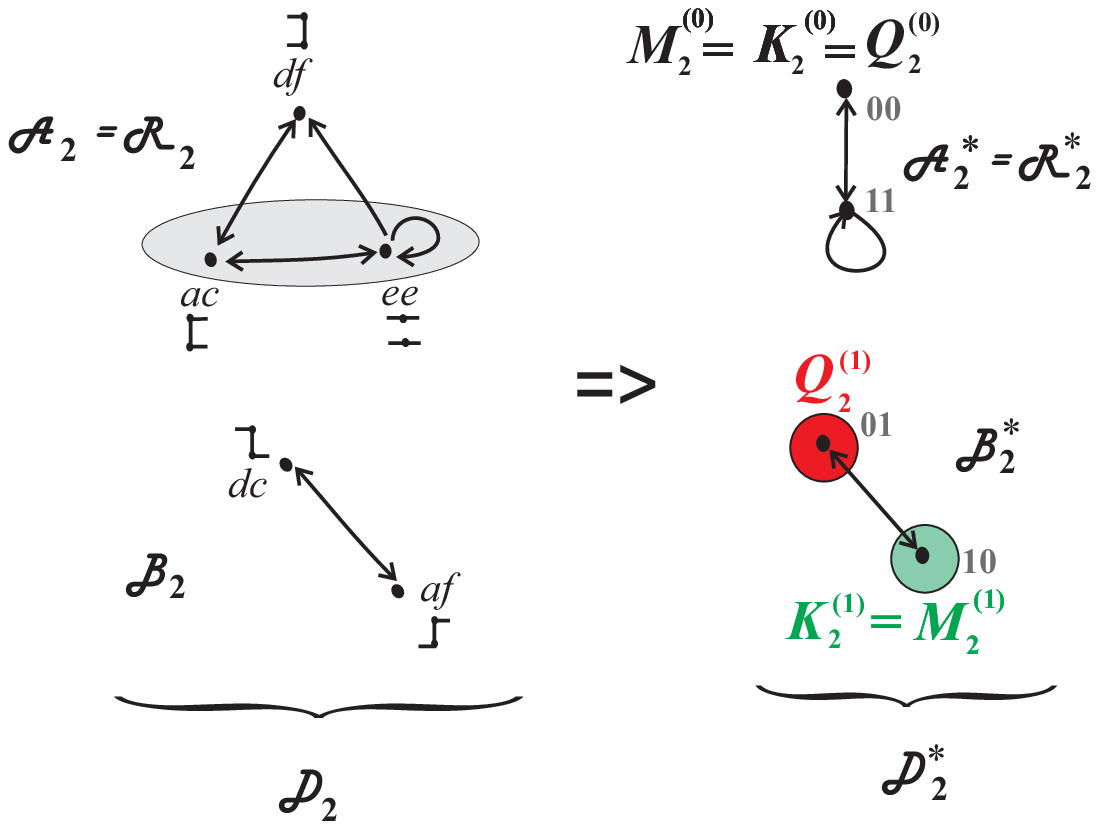}
\\ \ \vspace*{-18pt}
\end{center}
\caption{The  digraphs ${\cal D}_{2}$ and  ${\cal D}^*_{2}$}
\label{D2iD2zvezda}
\end{figure}

\begin{df} \label{def:2-factor}
The {\bf {\em outlet word}} of  a vertex $\alpha \equiv \alpha_1 \alpha_2 \ldots \alpha_m \in V({\cal D}_{m})$ is  the  binary word
$o(\alpha) \equiv o_1o_2 \ldots o_m$, where   $ \ds o_j
\stackrel{\rm def}{=} \left \{
\begin{array}{cc}{}
0, & \; \; if \; \; \alpha_j \in \{ b, d, f \} \\
1, & \; \; if \; \; \alpha_j \in \{ a, c, e \}
\end{array}
 \right. ,  \; \; \; 1 \leq j \leq m.
$
\\
$\overline{\alpha}  \stackrel{\rm def}{=}\overline{\alpha}_m \overline{\alpha}_{m-1} \ldots \overline{\alpha}_1 \in V({\cal D}_{m})$ and
 $\overline{o(\alpha)} \stackrel{\rm def}{=} o(\overline{\alpha})$
\end{df}

\begin{exm}  \label{exm:0}
The first and  fifth  column of the code matrix for the  2-factor in  Figure~\ref{Adinkras}  are the words $(ac)^2ab^4c(ac)^2$ and $e(df)^2db^2f(df)^2e$, respectively. Their outlet words are $1^50^41^5$  and $10^{12}1$.
\end{exm}

The   digraph \ $ {\cal D}^*_{m} \stackrel{\rm def}{=}(V({\cal D}^*_{m}), E({\cal D}^*_{m}))$ is obtained
by gluing all  the vertices from  \ $ V({\cal D}_{m})$ \  \ having the same corresponding outlet word
 and replacing all the arcs from  $ E({\cal D}_{m})$ starting from these glued vertices and ending
at    same vertex with only one arc (see  Figure~\ref{D2iD2zvezda}).
 In \cite{DjBD} it is proved that
  every binary word from $ \{ 0,1 \}^m$ except the word  $(01)^k0$ when  $m=2k+1$ ($k \in N $) belongs to  $V({\cal D}^*_m)$ and that
 the adjacency matrix  ${\cal T}^*_m$ of the digraph ${\cal D}^*_m$ is a symmetric binary matrix,
 i.e.   if $v \rightarrow w$, then $w \rightarrow v$, for all $v,w \in V({\cal D}^*_m)$  (for this reason we occasionally use the label $v \leftrightarrow  w$).
 Consequently, each  component of ${\cal D}^*_m$ is a strongly connected digraph.
 By implementation of the algorithm for obtaining the digraphs ${\cal D}^*_m$
  the  data for $m \leq 12$  gathered in \cite{DjBD}
   suggest that  all the components of  ${\cal D}^{*}_{m}$  except one  are bipartite digraphs
and that the  order of each component could be expressed in term of  binomial coefficients.
 In this paper,  we prove that these  assumptions are true. Moreover, for every component we give a characterisation of its set of vertices.

\begin{thm}  \label{conj:1}  (Conjecture  in \cite{DjBD}) \\
For each  $m\geq 2 $,  the digraph  ${\cal D}^{*}_m$ has exactly $\ds \left\lfloor    \ds \frac{m}{2} \right\rfloor + 1$ components, i.e.
 $\ds {\cal D}^{*}_m = {\cal A}^*_m    \cup $ $\ds  (\bigcup_{s=1}^{\left\lfloor    \frac{m}{2} \right\rfloor }{\cal B}^{*(s)}_m)$, where
 $\ds \mid V({\cal B}^{*(1)}_{m}) \mid  \geq
  \mid V({\cal B}^{*(2)}_{m}) \mid  \geq  \; \; \;  \ldots \; \; \;  \geq \mid V( {\cal B}^{*(\lfloor    m/2 \rfloor )}_{m}) \mid $
 and   ${\cal A}^*_m $ is the one containing $1^m$.
All the components  ${\cal B}^{*(s)}_{m}$ ($ 1 \leq s \leq \ds \left\lfloor    \ds \frac{m}{2} \right\rfloor $) are bipartite digraphs.

 If  $m$ is  odd, then $\ds \mid V({\cal B}^{*(s)}_{m}) \mid  =\ds  {m + 1 \choose   (m+1)/2 -s} \mbox{ \  and \ } \ds \mid V({\cal A}^{*}_{m}) \mid =  \ds {m  \choose (m-1)/2}.$

  If  $m$ is even, then $\ds \mid V({\cal B}^{*(s)}_{m}) \mid  = \ds 2 {m \choose   m/2 -s} \mbox{ \  and \  } \ds \mid V({\cal A}^{*}_{m}) \mid =  \ds {m \choose m/2}.$ \\
    The vertices $v$ and $\overline{v}$ belong to  the same component. When the component is bipartite they are placed in the same  class  iff $m$ is odd.
\end{thm}

\begin{exm}  \label{exm:2}
The digraph  ${\cal D}^{*}_4$ depicted in  Figure~\ref{D3_D4}
has $ \ds \left\lfloor  \frac{4}{2} \right\rfloor + 1=3$ components of which all but the one  with the loop are bipartite digraphs.
The cardinalities of its components are  $ \ds {4 \choose 4/2} = 6$, $ \ds 2 {4 \choose   4/2 -1}=8$ and $ \ds 2 {4 \choose   4/2 -2}=2$.
Note that  the vertices $v$ and $\overline{v}$ are placed in the different  classes (of different colors).
\end{exm}

\begin{figure}[htb]
\begin{center}
\includegraphics[width=0.8\columnwidth]{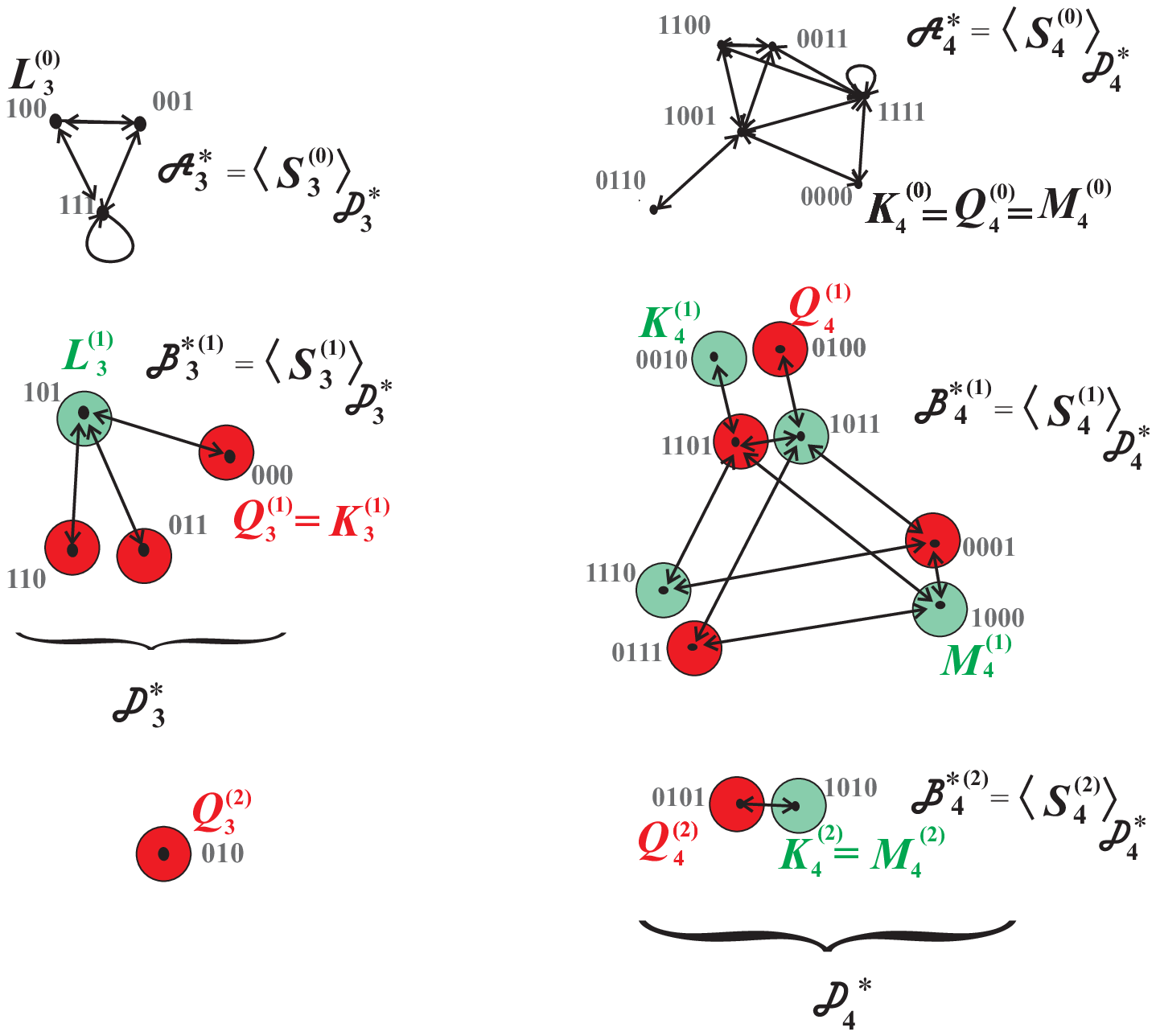}
\\ \ \vspace*{-18pt}
\end{center}
\caption{{\textbf{The digraphs  ${\cal D}_3^*$ and  ${\cal D}_4^*$  }}}
\label{D3_D4}
\end{figure}

The transfer matrix  ${\cal T}^*_m= [a_{ij}]$ which is used for the enumeration of  2-factors in the considered grid graphs is the adjacency (binary) matrix  of ${\cal D}_m^*$.

\begin{thm}   \label{thm:stara}(\cite{DjBD})\\
If $f_m^G(n)$ denotes the number of 2-factors  of the linear grid graph  $G$ of width $m$ with $m \cdot n$ vertices, then
 \\
  \parbox[t]{3cm}{$$f_m^G(n)=  \left \{ \begin{array}{rl}{} \vspace*{2.5cm}\end{array} \right. $$}
\parbox[t]{8cm}{$$ \begin{array}{rl}  \begin{array}{rl} a_{1,1}^{(n)}, & \mbox{ if  }\; \; G = RG , \\ \\
 tr(({\cal T}_m^*)^n) = \ds \sum_{v_i \in   V({\cal D}_m^*) } a_{i,i}^{(n)},  &                   \mbox{ if  }\; \; G= TkC,\\ \\
    tr({\cal P}_m^{*}  \cdot   ({\cal T}_m^*)^n) ,           &          \mbox{ if  }\; \; G = MS, \hspace*{0.2cm}\end{array}\end{array} $$}\\
 where      $v_1 \equiv 0^m$ (corresponding to the first row and first column of ${\cal T}^*_m$) and ${\cal P}_m^*$ is the permutation matrix of order $ \mid V({\cal D}^*_{m}) \mid $ which represents  the product  of all transpositions $(v,\overline{v})$ $(v,\overline{v} \in V({\cal D}_m^*))$.
 \end{thm} \noindent

\begin{exm}  \label{exm:1}
The 2-factor in  Figure~\ref{Adinkras}   corresponds to the closed directed walk of length $14$ in ${\cal D}^*_{14}$ starting and finishing with $0^{14}$, namely, $0^{14} \rightarrow 1^50^41^5  \rightarrow 1^60^21^6 \rightarrow 1^50^41^5  \rightarrow 1^60^21^6  \rightarrow 10^{12}1  \rightarrow  \ldots    \rightarrow    1^60^21^6 \rightarrow  1^50^41^5  \rightarrow 0^{14} $.
\end{exm}

The  main goal in this paper is the first proof of Theorem~\ref{conj:1}.
\\ In Section 2, we introduce  the sets of binary words of length $m$, denoted by  $S_{m}^{(s)}$ ($0 \leq s \leq \lfloor m/2 \rfloor $) (and  their representatives)
by means of counting the numbers of zeros in even and odd positions. We prove that two representatives of different sets  $S_{m}^{(s_1)}$ and $S_{m}^{(s_2)},$ $s_1 \neq s_2$ can not be connected by a directed walk in  ${\cal D}_m^*$.

In Section 3, we prove that the subdigraphs of  ${\cal D}_m^*$ induced by $S_{m}^{(s)}$ for $0 \leq s \leq \lfloor m/2 \rfloor $ are strongly connected
(hence the  components of   ${\cal D}_m^*$)  and  moreover for $s \geq 1$ they are bipartite digraphs.
The cardinality of each class of these bipartite digraphs is determined. We prove that $v$ and $\overline{v}$ belong to the same set $S_m^{(s)}$, for any $v \in V({\cal D}_m^*)$. In this way, we complete the proof of Theorem~\ref{conj:1}.



\begin{center}
{\bf 2. PRELIMINARIES}
\end{center}

\begin{df} \label{def:ch}
For a binary word $x $ of length $m$ ($ m \in N$) we denote by $odd(x)$ ($even(x)$) the total number of 0's at odd (even) positions in $x$.
The difference  $odd(x) - even(x)$ is labeled as $Z(x)$.
\end{df}

In order to  prove Theorem~\ref{conj:1},  we introduce for each integer $m \geq 1$ the  sets of binary words of length $m$:
 $S_m^{(0)}, \; S_m^{(1)}= R_m^{(1)} \cup  G_m^{(1)}, \; S_m^{(2)}= R_m^{(2)} \cup  G_m^{(2)}, \ldots ,
\;  S_m^{(\lfloor m/2 \rfloor )}= R_m^{(\lfloor m/2 \rfloor )} \cup  G_m^{(\lfloor m/2 \rfloor )}$  in the following way:

\begin{df} \label{def:ch}
The set  $S_{m}^{(0)}$ ($m \in N$) consists of all the  binary $m$-words whose number of 0's at odd positions is equal to the  number of 0's at even  positions.
For  $1 \leq s \leq \lfloor m/2 \rfloor$,  $S_m^{(s)}\stackrel{\rm def}{=}  R_m^{(s)} \cup  G_m^{(s)}$  where
the words in   $R_{m}^{(s)}$  and $G_{m}^{(s)}$ are all the  binary words $x$ of the length $m$
for which  $ Z(x)  = s$ and  $ Z(x) = -s$, respectively. Additionally, if  $m$ is odd, then $R_m^{(\lceil m/2 \rceil)}\stackrel{\rm def}{=}\{ 0 (10)^{\lfloor m/2 \rfloor} \}$.
\end{df}

 Clearly, $x \in S_m^{(s)}$ iff  $\mid Z(x) \mid = s$ ($0 \leq s \leq \lfloor m/2 \rfloor $).
In this way  we define the collection of the sets $S_{m}^{(s)}, 0 \leq s \leq  \lfloor m/2 \rfloor$ which represents
  for even $m \in N$ a partition of the set  $\{0,1\}^m$, i.e.
$\ds \bigcup_{s=0}^{\lfloor m/2 \rfloor} S_m^{(s)}  =\{ 0,1\}^m = V({\cal D}_m^*)$.
When $m$ is odd,
the singleton set $R_m^{(\lceil m/2 \rceil)}$ has the only $m$-word $x$ which is not in set of vertices of ${\cal D}_m^*$  \cite{DjBD} and   $Z(x) = \lfloor m/2 \rfloor +1$.
Then, we have
$\ds \bigcup_{s=0}^{\lfloor m/2 \rfloor} S_m^{(s)}  = \{ 0,1\}^m \backslash R_m^{(\lceil m/2 \rceil)} = V({\cal D}_m^*)$, too .

\begin{exm}  \label{exm:n1}
It is easy to check that
\\
$S_1^{(0)} = \{1 \}$, $R_1^{(1)} = \{0 \}$ and $S_1^{(0)} = \{1 \} =  V({\cal D}_1^*)$

\vspace*{3mm}

\noindent
$S_2^{(0)}  = \{ 00, 11 \}$,
  $ S_2^{(1)}= R_2^{(1)} \cup  G_2^{(1)}$,  $R_2^{(1)} =\{ 01 \}$ and  $G_2^{(1)}  =\{ 10 \}$ and \\
  $S_2^{(0)} \cup   S_2^{(1)}= \{0,1 \}^2 =  V({\cal D}_2^*)$

\vspace*{3mm}

\noindent
$S_3^{(0)} =\{ 100, 111, 001  \} $, $ S_3^{(1)}= R_3^{(1)} \cup  G_3^{(1)}$,   $ R_3^{(1)}=\{ 000, 011, 110 \}$,   $G_3^{(1)}= \{ 101 \}$, \\ $R_3^{(2)}  = \{010 \}$ and
 $S_3^{(0)} \cup   S_3^{(1)}= \{0,1 \}^3 \backslash \{010 \}  =  V({\cal D}_3^*)$
 \vspace*{3mm}

\noindent
Also, perceive that  $R_{2k}^{(k)}$, $G_{2k}^{(k)}$, $G_{2k+1}^{(k)}$  and $R_{2k+1}^{(k+1)}$ are singleton sets, i.e.   $R_{2k}^{(k)}= \{ (01)^k \}$, $G_{2k}^{(k)}= \{ (10)^k \}$, $G_{2k+1}^{(k)}= \{ (10)^k1 \}$   and  $R_{2k+1}^{(k+1)}= \{ (01)^k0 \}$ ($k \in N$). \\ Additionally,  the word $1^m $ belongs to $S_{m}^{(0)}$ for any $m \in N$.
\end{exm}

Considering the entries of $S_{m}^{(s)}$ as the vertices of ${\cal D}_m^*$ we call
the words from  $R_{m}^{(s)}$ and $G_{m}^{(s)}$ \ {\em red} and {\em green vertices}, respectively.
The  only binary word of length $m$ that does not belong to $\ds \bigcup_{s=0}^{\lfloor m/2 \rfloor} S_m^{(s)} = V({\cal D}_m^*)$   is $(01)^k0 \in R_{m}^{(k+1)}$ where $m= 2k+1$. Still, we call it {\em red vertex} (by definition).

For an arbitrary  set of binary words $A$, the label  $1A$ (or  $0A$) denotes the set of all words $1a$ ($0a$), where $a \in A$.
Note that the following equalities hold for any binary word $x$:
\begin{prp} \label{prp:novi} $ $

 a)  $Z(0x) = -Z(x)+1$,

 b)  $Z(1x) = -Z(x)$ ,

 c)  $Z(00x) = Z(11x) = Z(x)$,

 d)  $Z(10x) = Z(x)-1$  and

 e)  $Z(01x) = Z(x)+1$.
\end{prp}

Observe that  a  vertex from $S_{m}^{(0)}$  is obtained by adding the  prefix $1$ to a  vertex from  $S_{m-1}^{(0)}$ or the  prefix $0$ to a red vertex from  $R_{m-1}^{(1)}$.
A red vertex from $R_{m}^{(s)}$, $1 \leq s \leq  \lfloor m/2 \rfloor$  is obtained by adding the  prefix $1$ to a green vertex from  $G_{m-1}^{(s)}$ or the  prefix $0$ to a green vertex from  $G_{m-1}^{(s-1)}$ (when $s \geq 2$) or to a vertex from $S_{m-1}^{(0)}$ (when $s=1$).
Similarly,  a green vertex from $G_{m}^{(s)}$, $1 \leq s \leq  \lfloor m/2 \rfloor$  is obtained by adding the  prefix $1$ to a red vertex from  $R_{m-1}^{(s)}$ or the  prefix $0$ to a red vertex from  $R_{m-1}^{(s+1)}$ when $s< \lfloor m/2 \rfloor$.

 In what follows, we will prove that the subdigraphs of ${\cal D}_m^*$ induced by  $S_{m}^{(s)} $  ($0 \leq s \leq \lfloor m/2 \rfloor $),  labeled as $ \langle S_{m}^{(s)} \rangle_{{\cal D}_m^*}$, are $ {\cal A}_m^*$ (the component which contains $1^m$), ${\cal B}^{*(1)}_{m}$,  ${\cal B}^{*(2)}_{m}, \ldots , $ $ {\cal B}^{*(\lfloor     m/2 \rfloor )}_{m}$ (which satisfy  \  $ \mid V({\cal B}^{*(1)}_{m}) \mid  \geq
  \mid V({\cal B}^{*(2)}_{m}) \mid  \geq  \ldots $  $\geq \mid V( {\cal B}^{*(\lfloor    m/2 \rfloor )}_{m}) \mid $), respectively, and that each  ${\cal B}_m^{*(s)}$ is the bichromatic (hence bipartite) digraph $(R_m^{(s)}, G_m^{(s)})$ ($s \geq 1$).

\begin{df} \label{def:nova1}
 We refer to
the zero-word  $Q_{2k}^{(0)} \stackrel{\rm def}{=} 0^{2k} \in S_{2k}^{(0)} $, the words   $Q_{2k}^{(s)} \stackrel{\rm def}{=} (01)^s0^{2k-2s} \in R_{2k}^{(s)}$
 ($1 \leq s \leq k $) and
$Q_{2k+1}^{(s)} \stackrel{\rm def}{=} (01)^{s-1}0^{2k-2s+3} \in R_{2k+1}^{(s)}$ ($1 \leq s \leq k+1$) as  the {\bf {\em queens}}.
The words  $(10)^s0^{2k-2s}\in S_{2k}^{(s)} $ ($0 \leq s \leq k $) are  called  the  {\bf {\em maidens}} and labeled as $M_{2k}^{(s)}$.
Similarly, the word  $L_{2k+1}^{(k)} \stackrel{\rm def}{=}(10)^{k}1 \in S_{2k+1}^{(k)}$  and the words
 $L_{2k+1}^{(s)} \stackrel{\rm def}{=}(10)^{s+1}0^{2k-2s-1} \in S_{2k+1}^{(s)}$ ($0 \leq s < k $) are  called the  {\bf {\em court ladies}}.
\end{df}

 Note that $M_{2k}^{(0)}\equiv Q_{2k}^{(0)} $,
 $M_{2k}^{(s)} \in G_{2k}^{(s)}$ and $L_{2k+1}^{(s)} \in G_{2k+1}^{(s)}$ where $s \geq 1$.
In this way, when $s \geq 1$ we have provided  the representatives for the red and green sets: the queens $Q_{m}^{(s)}$ ($m \in N$) for the first ones,  and the maidens  $M_{2k}^{(s)}$ ($k \in N$) and the court ladies  $L_{2k+1}^{(s)}$ ($k \in N$) for the second ones. In these cases, we treat the queen $Q_{m}^{(s)}$ as the  main representative for the  entire set $S_{m}^{(s)}$. Find that the only set $S_{m}^{(s)}$ without a queen is $S_{m}^{(0)}$ when $m$ is odd. In this case, the court lady $L_{m}^{(0)}  \equiv10^{m-1}$ becomes the  main representative for $S_{m}^{(0)}$, while  when  $m$ is even this role takes the queen $Q_{m}^{(0)}\equiv M_{m}^{(0)} = 0^m$.

If we  add the prefix $0$  to the maiden  $M_{2k}^{(s)} \in S_{2k}^{(s)}$ ($0 \leq s \leq k $), she becomes the queen in $R_{2k+1}^{(s+1)}$. If we add the prefix $1$ to the queen $Q_{2k}^{(s)}$ ($0 \leq s \leq k $), she becomes  the court lady $L_{2k+1}^{(s)} \in S_{2k+1}^{(s)} $.
 Reversible ``aging'' process, i.e. ``rejuvenation'' arises during the  forming stage  of  the representatives for the  red and green subsets of $S_{2k}^{(s)}$, where $s \geq 1$ ($k \in N$). Namely, by adding $1$ as a prefix to the queen  $Q_{2k-1}^{(s)} \in R_{2k-1}^{(s)} $, where $ s \geq 1$,  we obtain the maiden $M_{2k}^{(s)} \in G_{2k}^{(s)}$.   Additionally, the queen $Q_{2k}^{(s)} \in R_{2k}^{(s)}$ ($1 \leq s \leq k  $) is obtained
  by adding $0$ as a prefix to  the court lady $L_{2k-1}^{(s-1)} \in S_{2k-1}^{(s-1)} $.
  \\  This  rule does not apply for  the queen from $S_{2k}^{(0)}$. In fact, she arises from the queen from  $R_{2k-1}^{(1)}$, i.e. $Q_{2k}^{(0)} = 0^{2k}= 0Q_{2k-1}^{(1)}  \in S_{2k}^{(0)}$.

\begin{df} \label{def:nova2}
The word $\overline{Q}_m^{(s)}  $ ($0 \leq s \leq \lfloor m/2 \rfloor $) is called the {\bf {\em king}} and  labeled as $K_{m}^{(s)}$.
\end{df}

 Clearly, for $s=0$ we have $K_{2k}^{(0)}  \equiv Q_{2k}^{(0)}  \equiv M_{2k}^{(0)} = 0^{2k}$ ($k \in N$).
When $s \geq 1$, in case  $m = 2k$ ($k \in N$), $K_{2k}^{(s)} = 0^{2k-2s}(10)^{s}  \in G_{2k}^{(s)} $, while  in case   $m = 2k+1$ ($k \in N$), $K_{2k+1}^{(s)}  \in R_{2k+1}^{(s)} $ because $K_{2k+1}^{(s)} = 0^{2k-2s+3}(10)^{s-1}  \in R_{2k+1}^{(s)}$. (When the maiden is   present in $S^{(s)}_{m}$ ($s \geq 1$), then the king takes ``the same side'' (color) as the maiden, but if the  court lady is  present in $S^{(s)}_{m}$, then he takes ``the opposite side'' (color) of  the  court lady.)

 \begin{lem}  \label{lem:2} The court lady $L_{2k-1}^{(0)} \in S_{2k-1}^{(0)}$ ($k \in N $) and  any queen $Q_{2k-1}^{(s)} \in S_{2k-1}^{(s)}$  ($1 \leq s   \leq k-1$)  can not be  connected  by a directed walk in ${\cal D}_{2k-1}^*$. The same is valid  for any two queens   $Q_{m}^{(s_1)}$ and $Q_{m}^{(s_2)}$ where $0 \leq s_1 < s_2 \leq \lfloor \frac{m}{2} \rfloor$ ($m \in N$).
 \end{lem}
\begin{proof}
Recall that  this statement is trivially  valid for the main representatives of  $S_{m}^{(s_1)}$ and $S_{m}^{(s_2)}$ ($0 \leq s_1 < s_2 \leq \lfloor \frac{m}{2} \rfloor$)
when   $s_1$ and  $s_2$    (i.e.  the numbers of their 1's) have   opposite parity.

 The remaining  cases are discussed below.
In all cases we give indirect proofs. Thus, we suppose the opposite, that in ${\cal D}_m^*$ there exists a directed walk $v_0 \rightarrow v_1 \rightarrow v_2 \rightarrow \ldots \rightarrow v_{k-1} \rightarrow v_k$ of length $k \in N$, where   $ v_0$ and $ v_k$ are the  main representatives of the two considered  sets (no matter in which direction).

The corresponding part of the grid for the directed  walk  from $v_1$ to $v_k$ has $m \cdot k$ vertices in the grid (see  Figure~\ref{NPm}). Note that this rectangular grid graph is bichromatic (we color  its vertices in gray and black). In this grid the directed walk $v_0 \rightarrow v_1 \rightarrow v_2 \rightarrow \ldots \rightarrow v_{k-1} \rightarrow v_k$ determines a spanning union of paths  (open paths and   cycles). In this union  each   cycle (if exists)  has the same number of  vertices of both colors.
 Further, when  $m\cdot k$ is even,  then the numbers  of gray and black vertices of the  grid $P_m \times P_k$  are the same.  Otherwise,  when  $m\cdot k$ is odd, these numbers
  differ by $1$, the all four corner vertices (with the degree $2$) are of the same color -  the one whose vertices are more numerous.

\begin{figure}[htb]
\begin{center}
\includegraphics[width=0.90\columnwidth]{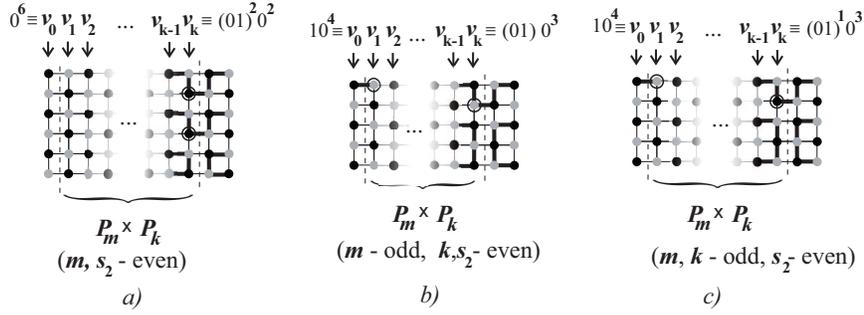}
\\ \ \vspace*{-18pt}
\end{center}
\caption{  The main representatives of  $S_{m}^{(0)}$ and $S_{m}^{(s_2)}$  ($ 1 \leq s_2 \leq \lfloor \frac{m}{2} \rfloor$) can not be  connected  by a directed walk in ${\cal D}_m^*$.}
\label{NPm}
\end{figure}

\underline{Case I: $s_1 = 0$, $m$ and $s_2 $ $(s_2 >0)$ are even.}
\\ The directed walk $v_0 \rightarrow v_1 \rightarrow v_2 \rightarrow \ldots \rightarrow v_{k-1} \rightarrow v_k$ determines
a spanning union of   $\ds \frac{s_2}{2}$ open paths and  a few  cycles or without them.
The end vertices of these open paths  (rounded in  Figure~\ref{NPm} a)           belong to the last column of  the  grid $P_m \times P_k$.
They  are of the  same color and  related to 1's  of the queen $v_k= Q_{m}^{(s_2)}= (01)^{s_2}0^{m-2s_2}$.
Consequently, the numbers of gray and black vertices in the considered grid graph $P_m \times P_k$ differ by $\ds \frac{s_2}{2}$.
But, these numbers must be equal because $m \cdot k$ is even. Contradiction.

\underline{Case II: $s_1 = 0$, $m$  is  odd and $s_2$ $(s_2 >0)$ is even.} \\
In the directed walk $v_0 \rightarrow v_1 \rightarrow v_2 \rightarrow \ldots \rightarrow v_{k-1} \rightarrow v_k$ we have  $v_0 = L_{m}^{(0)} =10^{m-1}$ and
$v_k = Q_{m}^{(s_2)} =(01)^{s_2 -1}0^{m-2s_2+2}$.  \\
If $k$ is even  (see  Figure~\ref{NPm} b), the end vertices of all $\ds \frac{s_2}{2}$  ($\ds \frac{s_2}{2} \geq 1$) open paths are of the same color which  implies that in the  grid graph $P_m \times P_k$ the number of vertices of this color is greater than the one of opposite color which  is impossible because $m \cdot k$ is even. \\
If $k$ is odd (see  Figure~\ref{NPm} c),  exactly one open path has end vertices of different colors.
The end vertices of the remaining  $\ds \frac{s_2-2}{2}$  ($\ds \frac{s_2-2}{2} \geq 0$) open paths are all of the same color. Without loss of generality, let us say this color is  black. We conclude that the number of black vertices is  greater than or equal to the number of gray vertices. On the other hand, since  $m \cdot k$ is odd, in the considered grid the number of gray vertices (among them are corner ones) must be greater by $1$ than the number of black vertices.  Contradiction.

\underline{Case III: $s_2 > s_1 > 0$ and $k$ is even. } \\
The considered  grid graph has the same number of vertices of each color ($m \cdot k$ is even). However, its spanning graph has
 $\ds \frac{s_1+s_2}{2}$ open paths, when $m$ is even, and  $\ds \frac{s_1+s_2}{2} -1$ open paths, when $m$ is odd.
 These open paths cover $\ds \frac{s_2 - s_1}{2}\geq 1$  more vertices of one color than  another (see Figure~\ref{NP} a). Contradiction.

\begin{figure}[htb]
\begin{center}
\includegraphics[width=0.80\columnwidth]{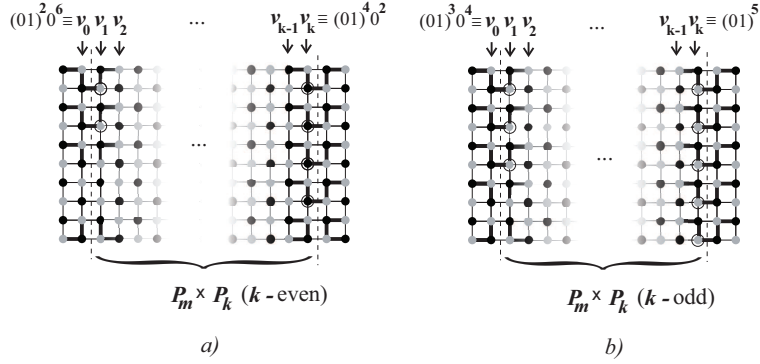}
\\ \ \vspace*{-18pt}
\end{center}
\caption{Two queens from   $S_{m}^{(s_1)}$ and $S_{m}^{(s_2)}$  ($ 1 \leq s_1 \leq s_2 \leq \lfloor \frac{m}{2} \rfloor$) can not be  connected  by a directed walk in ${\cal D}_m^*$.}
\label{NP}
\end{figure}

\underline{Case IV: $s_2 > s_1 > 0$ and $k$ is odd. } \\
If $k$ is odd, then all open paths have end vertices of the same color, let us say gray (see Figure~\ref{NP} b).
Consequently, gray vertices are more numerous. \\
 But, the upper  corner vertices are black. This   implies that the number of the black vertices is greater by $1$ (when $m$ is odd) or equal to  the number of gray vertices (when $m$ is even). Contradiction.
$\Box$ \end{proof}

We need a few assertions that  can be easily  obtained from the definition of 2-factor for the linear  grid graph  $G$.

\begin{prp} \label{r0}  For $v, w \in V({\cal D}_{m_1}^*)$ and $x, y \in V({\cal D}_{m_2}^*)$ ($m_1, m_2 \in N$)

a) if $v \leftrightarrow w$  and $x \leftrightarrow y$, then
$vx \leftrightarrow wy$ in ${\cal D}_{m_1 + m_2}^*$,

b) if $vx \leftrightarrow wy$ in ${\cal D}_{m_1 + m_2}^*$ and $v \leftrightarrow w$, then $x \leftrightarrow y$.
\end{prp}
\begin{prp} \label{r1} $   $ \\
A direct predecessor and successor for a word from $V({\cal D}_m^*)$ having

 a)  a prefix $0$ must have a prefix $1$,

 b) a prefix $01$ must have a prefix $10$.
\end{prp}

\begin{prp} \label{r2} $   $ \\
If a word $w$ from $V({\cal D}_m^*)$ has a prefix $00$ and $ w \leftrightarrow v$ for some $v \in V({\cal D}_m^*)$,
 then for the word $u$ obtained from $w$ by replacing that prefix with $11$ we have $ u \leftrightarrow v$, too.
\end{prp}
\begin{proof}
Observe a 2-factor of a linear  grid graph $G$ for which the words $w$ and $v$ are related to  adjacent columns in $G$.
The square of the  grid  corresponding to the prefix $00$ of $w$  has two vertical edges which   belong to the 2-factor.
If we replace these edges   with other two (horizontal) of the square, the obtained subgraph of $G$ is a 2-factor, too.
$\Box$ \end{proof}



\vspace{1.5cc}
\begin{center}
{\bf 3. PROOF OF  THEOREM~\ref{conj:1}}
\end{center}

\begin{lem}  \label{thm:1}
For $m \geq 2 $ and  $0 \leq s \leq \lfloor m/2 \rfloor $, the subdigraphs of  ${\cal D}^{*}_{m}$  induced by the sets $S_{m}^{(s)}$ are strongly connected.
Additionally, for $s\geq 1$ they are bipartite digraphs, i.e  $ \langle  S_{m}^{(s)} \rangle_{{\cal D}^{*}_{m}}  = (R_{m}^{(s)},  G_{m}^{(s)})$.
There is no edge which joins a vertex from $S_{m}^{(s)}$ to a vertex from $S_{m}^{(t)}$, where $ 0 \leq s < t \leq \lfloor m/2 \rfloor$, i.e.
$\ds {\cal D}^{*}_{m} = \bigcup_{s=0}^{\lfloor m/2 \rfloor} \langle  S_{m}^{(s)} \rangle_{{\cal D}^{*}_{m}} $.
\end{lem}
\begin{proof} The proof is by induction on $m$.
The statement for $m\leq 3$  (the base cases)  trivially holds (see Figure~\ref{D2iD2zvezda}, Figure~\ref{D3_D4} and Example~\ref{exm:n1}).
Let us  suppose that the statement is  true for all digraphs ${\cal D}^{*}_{w}$, when $w < m$  and prove it for $m  \geq 4$.

\vspace*{3mm}

\underline{$\langle S_{m}^{(s)} \rangle_{{\cal D}^{*}_{m}} $ ($0 \leq s \leq \lfloor m/2 \rfloor $) is strongly connected.} \\
We introduce the following sets: \\
$ {\cal O}_m^{(s)}  \stackrel{\rm def}{=} \{ 1v \mid v \in \{ 0,1 \}^{m-1}   \wedge \;  \mid Z(v) \mid = s \}$,
$ {\cal J}_m^{(s)}  \stackrel{\rm def}{=} \{ 00v \mid v \in \{ 0,1 \}^{m-2}   \wedge \;  \mid Z(v) \mid = s \}$ \\
and
$ {\cal K}_m^{(s)}  \stackrel{\rm def}{=} \{ 01v \mid v \in \{ 0,1 \}^{m-2}   \wedge \;  \mid Z(v) + 1 \mid = s \}$. Obviously,
$S_{m}^{(s)} = {\cal O}_m^{(s)} \cup {\cal J}_m^{(s)} \cup {\cal K}_m^{(s)}$.

Note that  $\langle {\cal O}_m^{(s)} \rangle_{{\cal D}^{*}_{m}}$ is isomorphic to  $\langle S_{m-1}^{(s)} \rangle_{{\cal D}^{*}_{m-1}}$.
Since the latter digraph is strongly connected by inductive hypothesis (abbreviated  I.H.),
 the same is  valid for the former digraph, too. Let us prove that each vertex $x \in
 {\cal J}_m^{(s)} \cup {\cal K}_m^{(s)}$ is a direct predecessor for a vertex from $ {\cal O}_m^{(s)}$.

 If $x = 00v \in  {\cal J}_m^{(s)} $ where $v \in S_{m-2}^{(s)}$ ($\mid Z(v) \mid = s $),
   then   by I.H. there exists $w \in S_{m-2}^{(s)}$  for which $v \leftrightarrow w$ in $\langle S_{m-2}^{(s)}\rangle_{{\cal D}^{*}_{m-2}}$ and $Z(w) = -Z(v)$, i.e.
   $\mid Z(w) \mid  =  \mid Z(v) \mid =s$.  Since   $00 \leftrightarrow 11$ in ${\cal D}^{*}_{2}$, Proposition~\ref{r0}a implies that
 $00v  \leftrightarrow 11w$ where $11w  \in {\cal O}_m^{(s)}$.

 If $x = 01v$ where $v  \in \{ 0,1 \}^{m-2} \wedge Z(v)= -1 \pm s $, then by I.H. there exists $w \in  \{ 0,1 \}^{m-2} $
 for which $v \leftrightarrow w$ (hence $Z(w)=-Z(v)$)  in $\langle S_{m-2}^{(\mid Z(v) \mid)} \rangle_{{\cal D}^{*}_{m-2}}$.
      Since $01 \leftrightarrow 10$ in ${\cal D}^{*}_{2}$ and $Z(10w)= -1+ Z(w)= -1 -Z(v)= \mp s$, using  Proposition~\ref{r0}a again, we conclude that $x= 01v \leftrightarrow 10w $ where  $ 10w \in   {\cal O}_m^{(s)}$.

\vspace*{3mm}

\underline{$\langle S_{m}^{(s)} \rangle_{{\cal D}^{*}_{m}} $ ($1 \leq s \leq \lfloor m/2 \rfloor $)  is the bipartite digraph $(R_{m}^{(s)},G_{m}^{(s)})$.}
\\
We need to prove that $R_{m}^{(s)}$ and $ G_{m}^{(s)}$ are stable sets, i.e.
  $\langle R_{m}^{(s)} \rangle_{{\cal D}^{*}_{m}}$ and $\langle G_{m}^{(s)} \rangle_{{\cal D}^{*}_{m}}$
are edgeless  digraphs. For this   purpose we consider the set $\{ x \in \{0,1\}^m \mid Z(x) = k\}$, where $k=s$ or $k=-s$.
It is    the union ${\cal O}_m^{(k)} \cup {\cal J}_m^{(k)} \cup {\cal K}_m^{(k)}$ where
\
$ {\cal O}_m^{(k)}  \stackrel{\rm def}{=}
  \{1v \mid  v\in \{0,1\}^{m-1}   \wedge   Z(v) = -k\} $, \\ $ {\cal J}_m^{(k)}  \stackrel{\rm def}{=}  \{00v \mid  v\in \{0,1\}^{m-2}   \wedge  \; Z(v) = k\}$   and
 $ {\cal K}_m^{(k)}  \stackrel{\rm def}{=}     \{01v \mid  v\in \{0,1\}^{m-2}  \wedge  \;  Z(v) = k-1\}$.
\\
 The digraph $\langle {\cal J}_m^{(k)} \cup {\cal K}_m^{(k)}  \rangle_{{\cal D}^{*}_{m}}$
  is an edgeless digraph
 because two  words with a prefix $0$  can not be neighbors  in ${\cal D}^{*}_{m}$ (Proposition~\ref{r1}a).
 \\   The digraph $\langle {\cal O}_m^{(k)}   \rangle_{{\cal D}^{*}_{m}}$ is also an empty digraph because
it is  isomorphic to $\langle \{ v\in \{0,1\}^{m-1}  \mid   Z(v) = -k\} \rangle_{{\cal D}^{*}_{m-1}}$, which is
$\langle G_{m-1}^{(s)}  \rangle_{{\cal D}^{*}_{m}}$ (for $k=s$) or $\langle R_{m-1}^{(s)}  \rangle_{{\cal D}^{*}_{m}}$ (for $k=-s$) (I.H.).

 It remains to prove that there are no arcs between $\langle {\cal O}_m^{(k)}   \rangle_{{\cal D}^{*}_{m}}$  and $\langle {\cal J}_m^{(k)}  \rangle_{{\cal D}^{*}_{1}}$, neither
between  $\langle  {\cal O}_m^{(k)}   \rangle_{{\cal D}^{*}_{m}}$  and $\langle  {\cal K}_m^{(k)} \rangle_{{\cal D}^{*}_{m}}$.
 Assume the  opposite:
 $$\ds (\exists x \in {\cal J}_m^{(k)})(\exists y \in {\cal O}_m^{(k)}) \; x \leftrightarrow y \; \mbox{   (\emph{Case I}) \ \   or \ \  }
  (\exists x \in {\cal K}_m^{(k)})(\exists y \in {\cal O}_m^{(k)}) \; x \leftrightarrow y  \; \mbox{    (\emph{Case II}).  }
 $$
\emph{ Case I:} $ x=00v \in  {\cal J}_m^{(k)}$, $y= 1w \in {\cal O}_m^{(k)} $ and $x \leftrightarrow y$. \\
Using Proposition~\ref{r2} we conclude that $11v \leftrightarrow 1w$, i.e.  $1v \leftrightarrow w$.
$Z(v) = Z(00v)=k$ implies  $Z(1v) = -k$. On the other hand $Z(w) = -k$ because of  $Z(1w) = k$. Since $1v \leftrightarrow w$, we conclude that $ \{v \mid  v\in \{0,1\}^{m-1}  \wedge Z(v) = -k\} $ is not a stable set. Contradiction with I.H.
\\
      \emph{ Case II:}  $x=01v \in  {\cal K}_m^{(k)} $ and $y \in {\cal O}_m^{(k)} = \{10w \mid  w\in \{0,1\}^{m-2}  \wedge Z(w) = k +1 \} \cup  \{11w \mid  w\in \{0,1\}^{m-2}  \wedge Z(w) = k \}  $ and $x \leftrightarrow y$.
\\ Taking in mind Proposition~\ref{r1}b and  Proposition~\ref{r0}b, we have $y =10w  $ and $v \leftrightarrow w$  where $Z(w) = k +1 $ and $Z(v) = k -1 $. Consequently, the digraphs inducted by
 $ \{v \in \{0, 1\}^{m-2}   \mid   Z(v) = k  +1 \}$  and
   $ \{v \in \{0, 1\}^{m-2}   \mid   Z(v) = k -1 \}$ are not disjoint in ${\cal D}_{m-2}^*$. Contradiction with I.H.

\vspace*{3mm}

\noindent
\underline{There is no arc connecting a vertex in $S_{m}^{(s)}$ to a vertex in $S_{m}^{(t)}$, where $ 0 \leq s < t \leq \lfloor m/2 \rfloor$.} \\
Assuming  the  opposite, from already stated strong connectivity of the digraphs $\langle S_{m}^{(s)} \rangle_{{\cal D}^{*}_{m}}$ and $\langle S_{m}^{(t)} \rangle_{{\cal D}^{*}_{m}}$
 we obtain that the main representatives of $S_{m}^{(s)}$ and $S_{m}^{(t)}$ are connected by a directed walk in ${\cal D}_m^*$
 which is in contrary to Lemma~\ref{lem:2}. Consequently, we have  $\ds {\cal D}^{*}_{m}= \bigcup_{s=0}^{\lfloor m/2 \rfloor} \langle  S_{m}^{(s)} \rangle_{{\cal D}^{*}_{m}} $.

\vspace*{3mm}

This completes the verification of the main statement for $m  \geq 4$. Hence, by induction the result is true for all integer $m \geq 2$, as required.
$\Box$ \end{proof}

\begin{lem}   \label{thm:2}
If  $v \in  {\cal S}^{(s)}_{m}$ ($0 \leq s \leq \lfloor m/2 \rfloor , m \in N$), then  the vertex $\overline{v}$ belongs to the same set $S^{(s)}_{m} $.
 Moreover, if $m$ is odd and $s>0$, then  the vertices $v$ and $\overline{v}$ are in the same class (of the same color).  If $m$ is even and  $s>0$, then the vertices $v$ and $\overline{v}$ are in the different  classes (of different colors).
\end{lem}
\begin{proof}
  Let $ v \rightarrow v_1 \rightarrow v_2 \rightarrow  \ldots \rightarrow v_t = Q^{(s)}_{m}$ be a directed  walk of length $t$ $ (t \geq 0)$ which connects an arbitrary  vertex  $v \in  S^{(s)}_{m}$ where $s \geq 1$ to the queen in the same set. Then, the directed walk $ \overline{v} \rightarrow \overline{v}_1 \rightarrow \overline{v}_2  \rightarrow \ldots \rightarrow \overline{v}_t = K^{(s)}_{m}$, as well as
  $\overline{v}_t = K^{(s)}_{m}  \rightarrow \overline{v}_{t-1} \rightarrow \overline{v}_{t-2}  \rightarrow \ldots \rightarrow \overline{v}_1 \rightarrow \overline{v}$,
  is  of the same length $t$.
When $m$ is odd, the king $K^{(s)}_{m} $  is colored  red, i.e. $K^{(s)}_{m} \in R^{(s)}_{m}$. Every directed walk between him and the queen $Q^{(s)}_{m} $ is of even length, so the same is valid for vertices $v$ and $ \overline{v}$. When $m$ is even, the king $K^{(s)}_{m}$ and the queen $Q^{(s)}_{m}$ are of different colors (i.e. in different classes). Therefore, the same holds for $v$ and $\overline{v}$.

For  $v \in  {\cal S}^{(0)}_{m}$,  $\overline{v}$  must belong to the same set.  Namely,
if we assume the opposite, i.e.  $\overline{v} \in  {\cal S}^{(s)}_{m}$, where $s \neq 0$, then
 $\overline{\overline{v}} =v \in  {\cal S}^{(s)}_{m}$,     which is impossible.
$\Box$ \end{proof}

\begin{lem}   \label{thm:3} If $1 \leq s \leq k $ ($k \in N$), then
\\  a)
$$\ds \mid R^{(s)}_{2k} \mid  = \mid G^{(s)}_{2k} \mid = \ds    {2k \choose  k-s}, \mbox{ \ }
\ds \mid S^{(s)}_{2k} \mid  = 2 \cdot {2k \choose  k-s}  \mbox{ \  and  \ }
\ds \mid S^{(0)}_{2k} \mid = \ds  {2k \choose   k}. $$
\noindent b)
\bc $\ds \mid R^{(s)}_{2k+1} \mid  =\ds  {2k +1 \choose   k-s+1}, \mbox{ \ }  \ds \mid G^{(s)}_{2k+1} \mid = \ds  {2k +1 \choose   k-s}, \mbox{ \ }
\ds \mid S^{(s)}_{2k+1} \mid = \ds  {2k +2 \choose   k-s+1}  \mbox{ \  and  } $ \\
$\ds \mid S^{(0)}_{2k+1} \mid  = \ds  {2k+1 \choose   k}.$ \ec
\end{lem}

\begin{proof}
Let $i$   be the number of 0's at odd position in a word $v \in V({\cal D}_m^*)$, i.e. $i=odd(v)$.
Using  the definition of the sets $S_m^{(0)}$, $R_m^{(s)}$ and $G_m^{(s)}$ ($1\leq s \leq k$)  and Vandermonde's convolution we analyze all the cases.
\begin{itemize}
\item \underline{$m=2k, v \in S_m^{(0)}$} \\
$\ds \mid S^{(0)}_{2k} \mid  = \sum_{i=0}^{k}  {k \choose  i} {k \choose  i}  = \sum_{i=0}^{k}  {k \choose  i} {k \choose  k-i} =   {2k \choose  k} $

\item \underline{$m=2k+1, v \in S_m^{(0)}$} \\
$\ds \mid S^{(0)}_{2k+1} \mid  = \sum_{i=0}^{k}  {k+1 \choose  i} {k \choose  i}  = \sum_{i=0}^{k}  {k+1 \choose  i} {k \choose  k-i} =   {2k+1 \choose  k} $

\item \underline{$m=2k, v \in S_m^{(s)}$} \\
$\ds \mid R^{(s)}_{2k} \mid  = \sum_{i=0}^{k}  {k \choose  i} {k \choose  i-s}  = \sum_{i=0}^{k}  {k \choose  i} {k \choose  k-i+s} =   {2k \choose  k+s} =   {2k \choose  k-s} $ \\
$\ds \mid G^{(s)}_{2k} \mid  = \sum_{i=0}^{k}  {k \choose  i} {k \choose  i+s}  = \sum_{i=0}^{k}  {k \choose  i} {k \choose  k-i-s} =   {2k \choose  k-s} $

\item \underline{$m=2k+1, v \in S_m^{(s)}$} \\
$\ds \mid R^{(s)}_{2k+1} \mid  = \sum_{i=0}^{k}  {k+1 \choose  i} {k \choose  i-s}  = \sum_{i=0}^{k}  {k+1 \choose  i} {k \choose  k-i+s} =   {2k+1 \choose  k+s} =   {2k+1 \choose  k-s+1} $ \\
$\ds \mid G^{(s)}_{2k+1} \mid  = \sum_{i=0}^{k}  {k+1 \choose  i} {k \choose  i+s}  = \sum_{i=0}^{k}  {k+1 \choose  i} {k \choose  k-i-s} =   {2k+1 \choose  k-s}. $
\end{itemize}
$\Box$ \end{proof}

Now, from Lemma~\ref{thm:1} we conclude that the components of the transfer digraph ${\cal D}^{*}_{m}$ are $ \langle S_{m}^{(s)} \rangle_{{\cal D}^{*}_{m}} $ where $0 \leq s \leq \lfloor m/2 \rfloor$.
Note that the digraph  $ \langle S_{m}^{(0)} \rangle_{{\cal D}^{*}_{m}} $ is the component which  contains the vertex $1^m $. It implies
that $ \langle S_{m}^{(0)} \rangle_{{\cal D}^{*}_{m}} = {\cal A}_m^* $.
Having  in mind that
 $\ds {\cal D}^{*}_m = {\cal A}^*_m    \cup $ $\ds  (\bigcup_{s=1}^{\left\lfloor    \frac{m}{2} \right\rfloor }{\cal B}^{*(s)}_m) = $
 $\ds \bigcup_{s=0}^{\lfloor m/2 \rfloor} \langle  S_{m}^{(s)} \rangle_{{\cal D}^{*}_{m}} $ and
 $\mid S_m^{(s_1)}\mid > \mid S_m^{(s_2)}\mid $ for $1 \leq s_1 < s_2 \leq \lfloor m/2 \rfloor$, we conclude that
 $ {\cal B}^{*(s)}_{m}  = \langle S_{m}^{(s)} \rangle_{{\cal D}^{*}_{m}} $  for all $s=1,2, \ldots ,  \lfloor m/2 \rfloor$.
Consequently, all the components  ${\cal B}^{*(s)}_{m}$ ($ 1 \leq s \leq \ds \left\lfloor    \ds  m/2  \right\rfloor $) are bipartite digraphs.
Lemma~\ref{thm:2} and Lemma~\ref{thm:3} further  complete the proof of   the Theorem~\ref{conj:1}.

\hspace*{5mm}



\vspace{1.5cc}
\begin{center}
{\bf ACKNOWLEDGEMENTS}
\end{center}

The authors are indebted to the anonymous referees  for their valuable
suggestions and helpful comments which  improved the clarity of the presentation.
The authors would like to express their  gratitude to Roddy Bogawa
 for his meticulous reading of the first draft of the manuscript and on many useful suggestions.

This work  was  supported by  the Ministry of Education, Science and Technological Development of the   Republic of Serbia (Grants No. 451-03-9/2022-14/200125, 451-03-68/2022-14/200156)
and the Project of the Department for fundamental disciplines in technology, Faculty of Technical Sciences, University of Novi Sad "Application of general disciplines in technical and IT sciences".


\vspace*{0.5cm}

\noindent Faculty of Technical Sciences,
  University of Novi Sad,
  Novi Sad, Serbia\\
     E-mail: jelenadjokic@uns.ac.rs  \\
E-mail: ksenija@uns.ac.rs (corresponding author)

\vspace*{0.5cm}

 \noindent
  Dept.\ of Math.\ \&\ Info.,
  Faculty of Science,
  University of Novi Sad,
  Novi Sad, Serbia \\
   E-mail: olga.bodroza-pantic@dmi.uns.ac.rs

\end{document}